\documentclass[12pt]{amsproc}
\newtheorem{theorem}{\sc Theorem}[section]
\newtheorem{lemma}[theorem]{\sc Lemma}

\newtheorem{corollary}[theorem]{\sc Corollary}

\newtheorem{remark}[theorem]{\sc Remark}
\newtheorem{conjecture}[theorem]{\sc Conjecture}

\begin{document}

\title[Profinite groups in which  centralizers are procyclic]
{Profinite groups in which  centralizers are virtually procyclic}

\author{Pavel Shumyatsky}
\address{Department of Mathematics, University of Brasilia, Brazil}
\email{pavel@unb.br, pz@mat.unb.br}
\author{Pavel Zalesskii}

\thanks{This research was supported by CNPq and FAPDF}

\maketitle

\begin{abstract}
The article deals with profinite groups in which centralizers are virtually procyclic. Suppose that $G$ is a profinite group such that the centralizer of every nontrivial element is virtually torsion-free while the centralizer of every element of infinite order is virtually procyclic. We show that $G$ is either virtually pro-$p$ for some prime $p$ or virtually torsion-free procyclic (Theorem \ref{main1}). The same conclusion holds for profinite groups in which the centralizer of every nontrivial element is virtually procyclic (Theorem \ref{main2}); moreover, if $G$ is not  pro-$p$, then $G$ has finite rank.
\end{abstract}

\section{Introduction}

Centralizers play a fundamental role in group theory. In particular, centralizers were of crucial importance in the development of the theories of finite and locally finite groups (cf. \cite{go}, \cite{kewe}, and \cite{hartley}). On the other hand, results on centralizers in profinite groups are still fairly scarce.

In this paper our theme is as follows. Given an  information on centralizers of nontrivial elements of a profinite group $G$, what can be deduced about the structure of $G$? 

This study started in \cite{ppz} where the authors handled profinite groups with abelian centralizers. Later the work was extended by the first author to profinite groups in which the   centralizers are pronilpotent \cite{acn}.

The next condition on centralizers comes from hyperbolic groups that are known to have very restricted centralizes. Namely,  a residually finite hyperbolic group is virtually torsion free (see  \cite[Theorem 5.1]{KW-00}) and the centralizer of an element of infinite order of it is virtually cyclic (see  \cite[Proposition 3.5]{Mih}). Imposing this property on the centralizers of elements of profinite groups we show that a `hyperbolic-like' profinite group must be either virtually pro-$p$ or virtually procyclic.  Note that in general a procyclic group is not necessarily virtually torsion free (cf. Remark \ref{888}). 

\begin{theorem}\label{main1} Let $G$ be a profinite group such that the centralizer of every nontrivial element is virtually torsion-free while the centralizer of every element of infinite order is virtually procyclic. Then  $G$ is either virtually pro-$p$ for some prime $p$ or virtually torsion-free procyclic.
\end{theorem}

Having dealt with Theorem \ref{main1}, it was natural to look at profinite groups  in which the centralizer of every nontrivial element is virtually procyclic.

\begin{theorem}\label{main2} Let $G$ be a profinite group in which the centralizer of every nontrivial element is virtually procyclic. Then  $G$ is either virtually pro-$p$ for some prime $p$ or virtually procyclic.
\end{theorem}

Recall that a profinite group $K$ is said to have finite (Pr\"ufer) rank $r$ if every subgroup of $K$ can be generated by $r$ elements.
 The next corollary can be easily deduced from Theorem \ref{main2}.

\begin{corollary}\label{corr} Let $G$ be a profinite group in which the centralizer of every nontrivial element is virtually procyclic and suppose that $G$ is not a pro-$p$ group. Then  $G$ has finite rank.
\end{corollary}

In view of the above corollary the following conjecture seems plausible.

\begin{conjecture} Let $G$ be a profinite group in which the centralizer of every nontrivial element has finite rank. Suppose that $G$ is not a pro-$p$ group. Then  $G$ has finite rank.
\end{conjecture}

Note that the results obtained in \cite{acn} show that the conjecture is correct if we additionally assume that the centralizers in $G$ are pronilpotent.

The proofs of both Theorem \ref{main1} and Theorem \ref{main2} depend on the classification of finite simple groups.

\section{Preliminaries}

We start this section by quoting  the following famous theorem of Zelmanov \cite{ze}.

\begin{theorem}\label{torsion} Each compact torsion group is locally finite.
\end{theorem}

Recall that according to the Hall-Kulatilaka theorem \cite{haku} each infinite locally finite group has an infinite abelian subgroup. Combining this with Theorem \ref{torsion} we deduce

\begin{theorem}\label{infab} Each infinite locally finite, or profinite, group has an infinite abelian subgroup.
\end{theorem}

If $A$ is a group of automorphisms of a group $G$, the subgroup generated by all elements of the form $g^{-1}g^\alpha$ with $g\in G$ and $\alpha\in A$ is denoted by $[G,A]$. It is well-known that the subgroup $[G,A]$ is an $A$-invariant normal subgroup in $G$. We write $C_G(A)$ for the centralizer  of $A$ in $G$.

The next lemma is a list of useful facts on coprime actions. Here $|K|$ means the order of a profinite group $K$ (see for example \cite{rz} for the definition of the order of a profinite group). For finite groups the lemma is well known (see for example \cite[Ch.~5 and 6]{go}). For infinite profinite groups the lemma follows from the case of finite groups and the inverse limit argument (see \cite[Proposition 2.3.16]{rz} for a detailed proof of item (iii)). As usual, $\pi(G)$ denotes the set of prime divisors of the order of $G$. By an automorphism of a profinite group we always mean a continuous automorphism.

\begin{lemma}\label{cc}
Let  $A$ be a profinite group of automorphisms of a profinite group $G$ such that $(|G|,|A|)=1$. Then
\begin{enumerate}
\item[(i)] $G=[G,A]C_{G}(A)$.
\item[(ii)] $[G,A,A]=[G,A]$. 
\item[(iii)] $C_{G/N}(A)=NC_G(A)/N$ for any $A$-invariant normal subgroup $N$ of $G$.
\item[(iv)] $G$ contains an $A$-invariant Sylow $q$-subgroup for each prime $q\in\pi(G)$.
\end{enumerate}
\end{lemma}

Throughout the paper we denote by $\langle X\rangle$ the (sub)group generated by the set $X$.

\begin{lemma}\label{resh} Let $A$ be an infinite procyclic group acting coprimely on a profinite group $G$. Let $A_i=A^{i!}$ denote the subgroup generated by the $i!$th powers of elements of $A$. Set $G_i=C_G(A_i)$. For each $A$-invariant open normal subgroup $N$ of $G$ there is $i$ such that $G=NG_i$. If there is an index $j$ such that $G_j=G_{j+k}$ for each $k=1,2\dots$, then $G_j=G$.
\end{lemma}
\begin{proof} Let $N$ be an $A$-invariant open normal subgroup of $G$. Obviously, $A$ induces a finite group of automorphisms of $G/N$. Thus, there exists $i$ such that $A_i$ acts on $G/N$ trivially. In view of Lemma \ref{cc} (iii) it follows that $G=NG_i$.

Now suppose that there is an index $j$ such that $G_j=G_{j+k}$ for each $k=1,2\dots$. Suppose that $G_j\neq G$. Then there exists an $A$-invariant open normal subgroup $N$ of $G$ such that $NG_j\neq G$. On the other hand, by the above, there is an index $i$ such that $G=NG_i$. Note that we have containments $$G_1\leq G_2\leq\dots\leq G_u\leq G_{u+1}\leq\dots.$$ Since by assumptions $G_j=G_{j+k}$ for each $k=1,2\dots$, it follows that $G_i\leq G_j$ and so $NG_j=G$, a contradiction.
\end{proof}

The next lemma uses the well-known corollary of the classification of finite simple groups that the order of any nonabelian finite simple group is divisible by 3 or 5.

\begin{lemma}\label{prost} Let $H$ be a profinite group in which for all odd primes $p\in\pi(G)$ the Sylow $p$-subgroups are finite. Then $H$ is virtually prosoluble.
\end{lemma}
\begin{proof} Suppose that the lemma is false. Choose an open $\{3,5\}'$-subgroup $K\leq G$. We see that $K$ is prosoluble, a contradiction.
\end{proof}

Throughout the paper we write $O_\pi(G)$ for the maximal normal pro-$\pi$ subgroup of the group $G$. As usual, the additive group of $p$-adic integers is denoted by $\Bbb Z_p$. Recall that the group of automorphisms of $\Bbb Z_p$ is isomorphic to $\Bbb Z_p\oplus C_{p-1}$ if $p\geq3$ and $\Bbb Z_2\oplus C_2$ if $p=2$ (see for example \cite[Theorem 4.4.7]{rz}). Here $C_n$ stands for the cyclic group of order $n$. Consequently, the group of automorphisms of any procyclic group is abelian. 

\begin{lemma}\label{prfi} Let $G$ be a virtually procyclic profinite group. The second derived group $G''$ is finite.
\end{lemma}
\begin{proof} Let $K$ be an open normal procyclic subgroup of $G$. Observe that $G'\leq C_G(K)$. It follows that the centre of $G'$ has finite index, whence by Schur's theorem \cite[Theorem 10.1.4]{rob} $G''$ is finite. 
\end{proof}

\begin{lemma}\label{008} Let $G$ be a profinite group whose Sylow 2-subgroups are virtually procyclic. Then $G$ is virtually prosoluble.
\end{lemma}
\begin{proof} If the Sylow 2-subgroups of $G$ are finite, then $G$ has an open normal 2$'$-subgroup which is prosoluble by the Feit-Thompson theorem \cite{fetho}. We therefore assume that the Sylow 2-subgroups of $G$ are infinite. Let $P$ be a Sylow 2-subgroup and $K$ a  procyclic normal subgroup of $P$. Choose an open normal subgroup $U$ of $G$ such that $U\cap P\leq K$.  Then all Sylow 2-subgroups of $U$ are procyclic.  Note that finite groups with cyclic Sylow 2-subgroups are soluble \cite[Theorem 7.6.1]{go} and therefore $U$ is prosoluble.  The proof is complete.
\end{proof}

Recall that the Fitting height of a finite soluble group $G$ is the length $h(G)$ of a shortest series of normal subgroups all of whose quotients are nilpotent. By the Fitting height of a prosoluble group $G$ we mean the length $h(G)$ of a shortest series of normal subgroups all of whose quotients are pronilpotent. Note that in general a prosoluble group does not necessarily have such a series. The parameter $h(G)$ is finite if, and only if, $G$ is an inverse limit of finite soluble groups of bounded Fitting height. The next lemma follows from the corresponding result on finite soluble groups (see for example \cite{cajaspi}).
\begin{lemma}\label{1000} Let $p$ be a prime, $h$ a positive integer, and $G$ a prosoluble group having a Hall $p'$-subgroup $K$ such that $h(K)=h$. Then $h(G)$ is finite and bounded in terms of $h$ only.
\end{lemma}

\section{Auxiliary results}

In this section we develop the common part of the proofs of Theorem \ref{main1} and Theorem \ref{main2}. We start though with two remarks underlining the difference between the theorems -- the fact that abstract locally finite subgroups in a group satisfying the hypothesis of Theorem \ref{main1} are finite while those in Theorem \ref{main2} may be infinite.

\begin{remark}\label{777} Because of Theorem \ref{infab} in a profinite group with virtually torsion-free centralizers all abstract locally finite subgroups are finite. Thus, if $G$ is as in Theorem \ref{main1} then the abstract locally finite subgroups of $G$ are finite.
\end{remark}

\begin{remark}\label{888} If $G$ is as in Theorem \ref{main2}, that is, a profinite group in which the centralizer of every nontrivial element is virtually procyclic, then $G$ can have infinite locally finite subgroups. On the other hand, it is easy to see that for any prime $p$ the abstract locally finite $p$-subgroups of $G$ are finite.
\end{remark}

The next two lemmas provide key technical tools for the proofs of Theorem \ref{main1} and Theorem \ref{main2} respectively. Though the statements of the lemmas are different, the proofs are almost identical so we give a common proof for both.

\begin{lemma}\label{100} Let $G$ be as in Theorem \ref{main1}, that is, a profinite group in which the centralizer of every nontrivial element is virtually torsion-free while the centralizer of every element of infinite order is virtually procyclic. Suppose that $G=HA$, where $A\cong\Bbb Z_p$ and $H$ a normal pro-$p'$ subgroup. Then $G$ is virtually procyclic.
\end{lemma}

\begin{lemma}\label{1002} Let $G$ be as in Theorem \ref{main2}, that is, a profinite group in which the centralizer of every nontrivial element is virtually procyclic. Suppose that $G=HA$, where $H$ is a normal pro-$p$ subgroup and $A$ an infinite procyclic pro-$p'$ subgroup. Then $G$ is virtually procyclic.
\end{lemma}

\begin{proof}[Proof of Lemma \ref{100} and Lemma \ref{1002}] Let $A_i=A^{i!}$ denote the subgroup generated by the $i!$th powers of elements of $A$. Set $H_i=C_H(A_i)$ for $i=1,2,\dots$. Each subgroup $H_i$ is virtually procyclic. Lemma \ref{prfi} shows that the second commutator subgroup ${H_i}''$ is finite. Therefore the abstract subgroup $H_0=\bigcup_i H_i''$ is locally finite. By virtue of Remark \ref{777} and Remark \ref{888}, $H_0$ is finite. Choose an open normal subgroup $M_1$ of $G$ such that $M_1\cap H_0=1$. It is sufficient to show that $M_1$ is virtually procyclic. We can work with the group $M_1A$ in place of $G$ and without loss of generality assume that $H_0=1$, that is, all centralizers $H_i$ are metabelian.

Suppose $N$ is an open normal $A$-invariant subgroup of $H$. Since $H/N$ is finite, there is an index $i$ such that $A_i$ acts trivially on $H/N$. Lemma \ref{cc} (iii) shows that $H=NH_i$. Recall that $H_i$ is metabelian. Hence $H/N$ is metabelian for any open normal $A$-invariant subgroup $N$ of $H$. We conclude that $H$ is metabelian. Due to the hypothesis on centralizers in $G$, it follows that $H'$ is an abelian virtually procyclic group. In the next paragraph we will deal with the case where $H'$ is finite. For now assume that $H'$ is infinite. Write $H'=K\times T$ where $T$ is a characteristic finite subgroup and $K$ a torsion-free procyclic subgroup. Choose an open normal subgroup $M_2$ of $G$ such that $M_2\cap T=1$. It is sufficient to show that $M_2$ is virtually procyclic. We can work with the group $M_2A$ in place of $G$ and without loss of generality assume that $T=1$, that is, $H'$ is torsion-free procyclic. Let $L$ be a subgroup of $H'$ isomorphic to $\Bbb Z_q$ for a prime $q$ (here $q\neq p$ in the case of Lemma \ref{100} and $q=p$ in the case of Lemma \ref{1002}). Since the group of automorphisms of $\Bbb Z_q$ does not contain infinite pro-$q'$ subgroups, a subgroup of finite index $B\leq A$ centralizes $L$. It is sufficient to show that $HB$ is virtually procyclic and so without loss of generality we assume that $A$ centralizes $L$. Then of course $[H,A]$ centralizes $L$. Recall that $H=[H,A]C_H(A)$. Note that $C_H(A)$ is virtually procyclic and so the index of the centralizer of $L$ in $C_H(A)$ is finite. Combining this with the facts that $[H,A]$ centralizes $L$ and $H=[H,A]C_H(A)$, we deduce that the index $[G:C_G(L)]$ is finite. Since $C_G(L)$ is virtually procyclic, $G$ is virtually procyclic, too. Thus, in the case where $H'$ is infinite the group $G$ is virtually procyclic.

We now deal with the case where $H'$ is finite. Passing to an open normal subgroup we may assume that $H$ is abelian. If $H$ is finite, the lemmas hold. So we assume that $H$ is infinite. It follows that $H$ is virtually procyclic and we can write $H$ as a direct product of a torsion-free procyclic subgroup and a finite subgroup. Choose an open normal subgroup $M_3$ of $G$ such that $M_3\cap H$ is torsion-free procyclic. As above, work with the group $M_3A$ in place of $G$ and without loss of generality assume that $H$ is torsion-free procyclic. Again, we consider a pro-$q$ subgroup $L$ of $H$ and see that its centralizer $C_G(L)$ has finite index in $G$ and is virtually procyclic. The proof is complete.
\end{proof} 

We can now reduce both Theorem \ref{main1} and Theorem \ref{main2} to questions on prosoluble groups.

\begin{lemma}\label{prosolu} Let $G$ be either as in Theorem \ref{main1} or as in Theorem \ref{main2}. Then $G$ is virtually prosoluble.
\end{lemma}

\begin{proof} Assume that $G$ is not virtually prosoluble. By virtue of the Feit-Thompson theorem, the Sylow 2-subgroups of $G$ are infinite. Moreover, by Lemma \ref{prost}, for at least one odd prime $q\in\pi(G)$ the Sylow $q$-subgroups of $G$ are infinite. 

Let $P$ be a Sylow 2-subgroup and $Q$ a Sylow $q$-subgroup of $G$. Since the subgroups $P$ and $Q$ are infinite, we can choose an infinite chain of open normal subgroups $G=N_1>N_2>\dots$ such that $P\cap N_i>P\cap N_{i+1}$ and $Q\cap N_i>Q\cap N_{i+1}$ for each $i=1,2,\dots$. 

Set $P_i=P\cap N_i$ and $K_i=N_G(P_i)$. The Frattini argument shows that $G=N_iK_i$ for each $i=1,2,\dots$. Suppose first that the Sylow $q$-subgroups $Q_i$ of $K_i$ are torsion for each $i=1,2,\dots$. Obviously, the subgroups $Q_i$ can be chosen in such a way that $Q_1\leq Q_2\leq\dots$. Let $Q_0=\cup_iQ_i$ and $N=\cap_iN_i$. The choice of the chain $G=N_1>N_2>\dots$ guarantees that $Q\cap N$ has infinite index in $Q$. Therefore $Q_0$ is an infinite (abstract) locally finite subgroup of $G$. In view of Remark \ref{777} (or Remark \ref{888}) this is a contradiction. Hence,  there is an idex $j$ such that $Q_j$ is not torsion. Choose an infinite procyclic subgroup $A\leq Q_j$. Since $A$ normalizes $P_j$, Lemma \ref{100} shows that $P_j$ is virtually procyclic. Now Lemma \ref{008} says that the subgroup $N_j$ is virtually prosoluble. Of course, this implies that also $G$ is virtually prosoluble. This completes the proof.
\end{proof}

\section{Theorem \ref{main1}}

In this section we prove Theorem \ref{main1}. Thus, throughout this section $G$ is as in Theorem \ref{main1}, that is, a profinite group in which the centralizer of every nontrivial element is virtually torsion-free while the centralizer of every element of infinite order is virtually procyclic.

\begin{lemma}\label{101} Suppose that $G$ is prosoluble, written as a product $G=PH$, where $P$ is a pro-$p$ subgroup and $H$ is a virtually procyclic pro-$p'$ subgroup. Then $G$ is either virtually pro-$p$ or virtually procyclic.
\end{lemma}
\begin{proof} First, we observe that $H$ is a Hall $p'$-subgroup of $G$. By Lemma \ref{1000}  $G$ has a finite series $$1=G_0\leq G_1\leq G_2\leq \dots\leq G_{s+1}=G$$ of characteristic subgroups such that each quotient $G_{u+1}/G_u$ is either a pro-$p$ group or a pro-$p'$ group.  Note that $G_s=P_0H_0$, where $P_0=P\cap G_s$ and $H_0=H\cap G_s$ (see for example \cite[Lemma 2.4]{2015annali}). Arguing by induction on the length of the above series we can assume that $G_s$ is either virtually pro-$p$ or virtually procyclic. If at least one of the groups $G_s$, $G/G_s$ is finite, the result is immediate so we assume that both $G_s$ and $G/G_s$ are infinite. 

Suppose first that $G_s$ is virtually pro-$p$. Since the Hall $p'$-subgroups of $G_s$ are finite, we can replace $G$ by its open subgroup whose intersection with $G_s$ is a pro-$p$ group and the intersection with $H$ is torsion-free procyclic. Therefore without loss of generality assume that $G_s$ is a pro-$p$ group and $G/G_s$ is a torsion-free procyclic pro-$p'$ group. Choose a nontrivial pro-$q$ subgroup $L$ in $G$, for $q\neq p$, and look at the product $G_sL$. Lemma \ref{100} shows that $G_s$ is virtually procyclic. Let $K$ be an infinite characteristic procyclic subgroup of $G_s$. Since pro-$p'$ groups of automorphisms of $\Bbb Z_p$ are finite, it follows that $C_G(K)$ has finite index in $G$ and is virtually procyclic, as required.

Now we need to deal with the case where $G_s$ is virtually procyclic. Suppose that $G/G_s$ is a pro-$p'$ group. Then as above $G/G_s$ is virtually procyclic, too and $G=G_sH$. Replacing $H$ by a subgroup of finite index we can assume that $G/G_s$ is procyclic. Let again $K$ be a characteristic open procyclic subgroup of $G_s$. If $K$ is virtually pro-$p'$, then the whole group $G$ is virtually pro-$p'$ and so $G$ is virtually procyclic since $H$ is open in $G$. Otherwise, assume that $K$ contains an infinite pro-$p$ subgroup $K_0$. Since pro-$p'$ groups of automorphisms of $\Bbb Z_p$ are finite, it follows that $C_G(K_0)$ has finite index in $G$. Taking into account that $C_G(K_0)$ is virtually procyclic, the result follows.

It remains to handle the situation where $G/G_s$ is a pro-$p$ group, in which case $G=G_sP$. Again we look at a characteristic open procyclic subgroup $K$ of $G_s$. If $K$ is virtually pro-$p$, then the whole group $G$ is virtually pro-$p$ and we have nothing to prove. Assume that $K$ is not virtually pro-$p$ and so $K$ contains an infinite pro-$q$ subgroup $K_1$ for some prime $q\neq p$. As before, we see that $C_G(K_1)$ has finite index in $G$. Taking into account that $C_G(K_1)$ is virtually procyclic, the result follows.
\end{proof}

\begin{lemma}\label{098} If $G$ is pronilpotent, then $G$ is either pro-$p$ for some prime $p$ or virtually torsion-free procyclic.
\end{lemma}
\begin{proof} Suppose that $G$ is not pro-$p$. Observe that for any $p\in\pi(G)$ the subgroup $O_{p'}(G)$ is virtually torsion-free procyclic. Therefore $G$ is a direct product of two virtually torsion-free procyclic subgroups of coprime orders, whence the lemma follows.
\end{proof} 

Recall that a Sylow basis in a group $L$ is a family of pairwise permutable Sylow $p_i$-subgroups $P_i$ of $L$, exactly one for each prime. The \emph{basis normalizer} of such Sylow basis in $L$ is the intersection of the normalizers of the $P_i$, that is, $\bigcap_i N_L(P_i)$. This subgroup is also known under the name of system normalizer. If $L$ is a profinite group and $T$ is a basis normalizer in $L$, then $T$ is pronilpotent and $L=\gamma_\infty(L)T$, where $\gamma_\infty(L)$ denotes the intersection of the terms of the lower central series of $L$. Furthermore, every prosoluble group $L$ possesses a Sylow basis and any two basis normalizers in $L$ are conjugate (see \cite[Prop.~2.3.9]{rz} and \cite[9.2]{rob}). 

We are now ready to complete the proof of Theorem \ref{main1}.
\begin{proof}[Proof of Theorem \ref{main1}] Recall that $G$ is a profinite group such that the centralizer of every nontrivial element is virtually torsion-free while the centralizer of every element of infinite order is virtually procyclic. We wish to prove that $G$ is either virtually pro-$p$ for some prime $p$ or virtually torsion-free procyclic. Because of Lemma \ref{prosolu} without loss of generality we can assume that $G$ is prosoluble. Choose a Sylow basis $P_1,P_2,\dots$ in $G$. Set $G=K_0$ and $K_{i+1}=\gamma_\infty(K_i)$ for $i=0,1,\dots$. Let $T_i$ denote the normalizer of the basis $P_1\cap K_i,P_2\cap K_i,\dots$ in $K_i$. Thus, the subgroups $T_i$ are pronilpotent and $T_i$ normalizes $T_j$ whenever $i\leq j$. Combining Remark \ref{777} with the property that $T_i$ normalizes $T_j$ whenever $i\leq j$ we derive that only finitely many of the subgroups $T_i$ are finite. Assume that the group $G$ is infinite and let $j$ be the minimal index such that $T_j$ is infinite. Since $T_j$ is infinite, Lemma \ref{098} guarantees that we can pick an infinite procyclic pro-$p$ subgroup  $A\leq T_j$. Recall that $T_j$ is a basis normalizer of $K_j$ and so $A$ normalizes a Sylow $q$-subgroup of $K_j$ for each $q\in\pi(T_j)$. Denote by $P$ a Sylow $p$-subgroup containing $A$ and by $H$ a Hall $p'$-subgroup normalized by $A$ so that $G=PH$. In view of Lemma \ref{100} $H$ is virtually procyclic. The theorem now follows from Lemma \ref{101}.
\end{proof}

\section{Theorem \ref{main2}} In this section we will prove Theorem \ref{main2}. Thus, throughout this section $G$ stands for a profinite group in which the centralizers of nontrivial elements are virtually procyclic.

\begin{lemma}\label{glll} If $G$ is prosoluble and $h(G)$ is finite, then $G$ is either virtually pro-$p$ for some prime $p$ or virtually procyclic.
\end{lemma}
\begin{proof} Use induction on $h=h(G)$. Suppose first that $h=1$, that is, $G$ is pronilpotent. Assume that $G$ is not pro-$p$. Observe that for any $p\in\pi(G)$ the subgroup $O_{p'}(G)$ is virtually procyclic. Therefore $G$ is a direct product of two virtually procyclic subgroups of coprime orders, whence the lemma follows.

We now suppose that $h\geq2$. Let $T$ be a basis normalizer in $G$ and $K=\gamma_\infty(G)$. Thus, $G=KT$. Since $T$ is pronilpotent and $h(K)=h-1$, by induction each of the subgroups $K$ and $T$ is either virtually pro-$p$ for some prime $p$ or virtually procyclic. We can choose an open normal subgroup $L$ in $G$ such that each of the subgroups subgroups $L_1=L\cap K$ and $L_2=L\cap T$ is either pro-$p$ for some prime $p$ or procyclic.

If both $L_1$ and $L_2$ are pro-$p$ for the same prime $p$, then $G$ is virtually pro-$p$ and we are done.

Suppose there are two different primes $p$ and $q$ such that $L_1$ is pro-$p$ while $L_2$ is pro-$q$. If $L_2$ has an infinite torsion-free procyclic subgroup, then by Lemma \ref{1002} $L_1$ is procyclic. In that case $C_G(L_1)$ has finite index in $G$ (because pro-$p'$ groups of automorphisms of a procyclic pro-$p$ group are finite). Since $C_G(L_1)$ is virtually procyclic, we are done. Otherwise, $L_2$ is torsion and so by Theorem \ref{torsion} locally finite. In view of Remark \ref{888} $L_2$ is finite and hence $G$ is virtually pro-$p$.

Now assume that $L_1$ is pro-$p$ while $L_2$ is procyclic. Let $B$ be the maximal pro-$p'$ subgroup of $L_2$. If $B$ is finite, then $G$ is virtually pro-$p$. We therefore assume that $B$ is infinite. Lemma \ref{1002} tells us that $L_1B$ is virtually procyclic. In particular, $C_B(L_1)$ has finite index in $B$. Observe that $C_B(L_1)$ commutes with both $L_1$ and $L_2$. Therefore $L\leq C_G(B)$. Note that $C_G(B)$ is virtually procyclic, whence $G$ is virtually procyclic, as required.

Suppose $L_1$ is procyclic while $L_2$ is pro-$p$. If $L_1$ has a nontrivial pro-$q$ subgroup for a prime $q\neq p$, we argue as above and deduce that $G$ is virtually procyclic because pro-$q'$ groups of automorphisms of a procyclic pro-$q$ group are finite. If $L_1$ is pro-$p$, the group $G$ is virtually pro-$p$.

It remains to consider the case where both $L_1$ and $L_2$ are procyclic while $L_1$ is not pro-$p$ (not even virtually pro-$p$). If $L_1$ has a finite subgroup $M$, we deduce that $G$ is virtually procyclic since $C_G(M)$ has finite index in $G$. Thus, assume that $L_1$ is torsion-free. Choose two subgroups $A_1,A_2\leq L_1$ where $A_1\cong\Bbb Z_p$ and $A_2\cong\Bbb Z_q$ for different primes $p$ and $q$. We can assume that both $C_G(A_1)$ and $C_G(A_2)$ have infinite index in $G$. Let $B_1$ and $B_2$ be the Sylow $p$- and $q$-subgroups of $L_2$, respectively. It follows that both $B_1$ and $B_2$ are infinite (since $C_G(A_1)$ and $C_G(A_2)$ have infinite index). Moreover, $B_1\cap C_G(A_1)=1$ and $B_2\cap C_G(A_2)=1$. Furthermore, $C_{B_2}(A_1)$ is infinite. Thus, $C_G(A_1)$ contains the subgroup $A_2C_{B_2}(A_1)$, which is not virtually procyclic. This contradiction completes the proof.
\end{proof}

We will require the celebrated theorem of Thompson \cite{tho} that says that if a finite soluble group $K$ admits a coprime automorphism $\phi$ of prime order, then $h(K)$ is bounded in terms of $h(C_K(\phi))$ only (see \cite{turull} for a survey on results of similar nature).

If $K$ is a locally finite group, we write $h(K)\leq h$ to mean that $K$ has a normal series of length at most $h$ all of whose factors are locally nilpotent. It is well-known that if $K$ is a locally finite-soluble group, then $h(K)\leq h$ if and only if $h(U)\leq h$ for every finite subgroup $U\leq K$.

We will now complete the proof of Theorem \ref{main2}. 
\begin{proof}[Proof of Theorem \ref{main2}] Assume that $G$ is infinite. Further, taking into account Lemma \ref{prosolu} without loss of generality we assume that $G$ is prosoluble. In view of Lemma \ref{glll} it is sufficient to show that $h(G)$ is finite. Choose a Sylow basis $P_1,P_2,\dots$ in $G$. Set $G=K_0$ and $K_{i+1}=\gamma_\infty(K_i)$ for $i=0,1,\dots$. Let $T_i$ denote the normalizer of the basis $P_1\cap K_i,P_2\cap K_i,\dots$ in $K_i$. Thus, the subgroups $T_i$ are pronilpotent and $T_i$ normalizes $T_j$ whenever $i\leq j$. Suppose that a Sylow $p$-subgroup of some $T_j$ is non-torsion. Then we can pick an infinite procyclic pro-$p$ subgroup  $A\leq T_j$. Recall that $T_j$ is a basis normalizer of $K_j$ and so $A$ normalizes a Hall $p'$-subgroup $H$ of $K_j$. In view of Lemma \ref{1002} $H$ is virtually procyclic. Therefore Lemma \ref{1000} shows that $h(G)$ is finite, as required.

Hence, we only need to deal with the case where all Sylow subgroups of the $T_i$ are torsion. In particular, because of Theorem \ref{torsion}, the Sylow subgroups of the $T_i$ are locally finite. Remark \ref{888} now implies that the Sylow subgroups of the $T_i$ are finite. Let $D_i$ denote the abstract subgroup generated by all Sylow subgroups of $T_i$. Observe that $D_i$ is locally finite and $D_i$ normalizes $D_j$ whenever $i\leq j$. Therefore the subgroup $D=\langle D_i;\ i=1,2,\dots\rangle$ is a locally finite-soluble group with finite Sylow subgroups. Choose an element $x\in D$ of prime order $p$. Since $D$ is residually finite and the Sylow $p$-subgroup of $D$ is finite, there is a normal $p'$-subgroup $N$ of finite index in $D$. By hypotheses $C_N(x)$ has a locally cyclic subgroup of finite index. Thus, there is $h$ such that $C_N(x)\leq h$. The Thompson Theorem \cite{tho} now tells us that there is $h_0$ such that $h(N)\leq h_0$. Since $N$ has finite index in $D$, we conclude that there is $h_1$ such that $h(D)\leq h_1$. Obviously, this implies that $D_{h_1}=1$ and we immediately obtain that $T_{h_1}=1$. Therefore $h(G)\leq h_1$, as requred. The proof is complete.
\end{proof}

We will now deduce Corollary \ref{corr}. Thus, assume that $G$ is a profinite group in which the centralizer of every nontrivial element is virtually procyclic and suppose that $G$ is not a pro-$p$ group. We wish to show that $G$ has finite rank. Obviously, this is correct if $G$ is virtually procyclic so, by virtue of Theorem \ref{main2}, we may assume that $G$ is virtually pro-$p$ for some prime $p$. We also assume that $G$ is infinite. Let $P$ be the largest normal pro-$p$ subgroup in $G$ and $a\in G$ an element of prime order $q\neq p$. Such an element $a$ exists since $G$ is not a pro-$p$ group. We now require the following theorem, due to Khukhro \cite{khukhro}.

\begin{theorem}\label{khu}
Let $K$ be a finite nilpotent group with an automorphism $\alpha$ of prime order $q$ such that $C_K(\alpha)$ has rank $r$. Then $K$ contains a characteristic subgroup $N$ such that $N$ has $q$-bounded nilpotency class and $K/N$ has $(q,r)$-bounded rank.
\end{theorem}
Using the standard inverse limit argument we deduce that if $K$ is a pronilpotent group admitting a coprime automorphism $\alpha$ of prime order $q$ such that $C_K(\alpha)$ has rank $r$, then $K$ contains a characteristic subgroup $N$ such that $N$ is nilpotent (with $q$-bounded nilpotency class) and $K/N$ has $(q,r)$-bounded rank.

Coming back to our group $G$, observe that $a$ induces an automorphism of $P$ of order dividing $q$. Since $C_P(a)$ is virtually procyclic, it follows that $C_P(a)$ has finite rank and therefore $P$ has a characteristic subgroup $N$ such that $N$ is nilpotent and $P/N$ has $(q,r)$-bounded rank. Since $N$ is nilpotent, it is contained in the centralizer of any nontrivial element of its centre. The centralizer is virtually procyclic and so both $N$ and $P/N$ have finite rank. We conclude that $P$ (and therefore $G$) has finite rank, as required.

\end{document}